\documentclass[4pt]{article} 

\usepackage[utf8]{inputenc} 
\usepackage{geometry} 
\usepackage{graphicx} 

\usepackage{amsmath}
\usepackage{amsfonts}
\usepackage{dsfont}
\usepackage{mathrsfs}
\usepackage{latexsym}
\usepackage{graphicx}
\usepackage{amssymb}
\usepackage{float}
\usepackage{epsfig}
\usepackage{epstopdf}
\usepackage{color,xcolor}
\usepackage{booktabs} 
\usepackage{array} 
\usepackage{paralist} 
\usepackage{verbatim} 
\usepackage{subfig} 

\newtheorem{theorem}{Theorem}[section]
\newtheorem{lemma}[theorem]{Lemma}
\newtheorem{corollary}[theorem]{Corollary}
\newtheorem{proposition}[theorem]{Proposition}

\newtheorem{problem}[theorem]{Problem}
\newtheorem{example}[theorem]{Example}

\newtheorem{remark}[theorem]{Remark}

\newcommand{\pmat}[1]{\begin{pmatrix}#1\end{pmatrix}}

\usepackage{fancyhdr} 
\usepackage{sectsty}
\allsectionsfont{\sffamily\mdseries\upshape} 

\usepackage[nottoc,notlof,notlot]{tocbibind} 
\usepackage[titles,subfigure]{tocloft} 

\title{\textbf{Graphs determined by their $A_{\alpha}$-spectra}}
\author{
Huiqiu Lin\thanks{Corresponding author. Supported by the National Natural Science Foundation of China (Nos.   11771141 and 11401211) and Fundamental Research
Funds for the Central Universities (No. 222201714049).}\\
{\small Department of Mathematics}\\[-0.8ex]
{\small East China University of Science and Technology}\\[-0.8ex]
{\small Shanhai 200237, P.R. China}\\
\emph{{\small \tt huiqiulin@126.com}}
\and
Xiaogang Liu\thanks{Supported by the National Natural Science Foundation of China (Nos. 11361033 and 11601431); the China Postdoctoral Science Foundation (No. 2016M600813), the Natural Science Foundation of Shaanxi Province (No. 2017JQ1019) and the Scientific Research Foundation of NPU (No. 3102016OQD029).}\\
{\small Department of Applied Mathematics}\\[-0.8ex]
{\small Northwestern Polytechnical University}\\[-0.8ex]
{\small Xi'an, Shaanxi 710072, P.R. China}\\
\emph{{\small \tt xiaogliu@nwpu.edu.cn}}
\and
Jie Xue\\
{\small Department of Computer Science and Technology}\\[-0.8ex]
{\small East China Normal University}\\[-0.8ex]
{\small Shanghai 200062, PR China}\\
\emph{{\small \tt jie\_xue@126.com}}
}
\date{} 

\begin{document}
\maketitle

\begin{abstract}
Let $G$ be a graph with $n$ vertices, and let $A(G)$ and $D(G)$ denote respectively the adjacency matrix and the degree matrix of $G$. Define
$$
A_{\alpha}(G)=\alpha D(G)+(1-\alpha)A(G)
$$
for any real $\alpha\in [0,1]$. The collection of eigenvalues of $A_{\alpha}(G)$ together with multiplicities are called the \emph{$A_{\alpha}$-spectrum} of $G$.
A graph $G$ is said to be \emph{determined by its $A_{\alpha}$-spectrum} if all graphs having the same $A_{\alpha}$-spectrum as $G$ are isomorphic to $G$.
We first prove that some graphs are determined by its $A_{\alpha}$-spectrum for $0\leq\alpha<1$, including the complete graph $K_m$, the star $K_{1,n-1}$, the path $P_n$, the union of cycles and the complement of the union of cycles,
the union of $K_2$ and $K_1$ and the complement of the union of $K_2$ and $K_1$, and the complement of $P_n$. Setting $\alpha=0$ or $\frac{1}{2}$, those graphs are determined by $A$- or $Q$-spectra. Secondly, when $G$ is regular, we show that $G$ is determined by its $A_{\alpha}$-spectrum if and only if the join $G\vee K_m$ is determined by its $A_{\alpha}$-spectrum for $\frac{1}{2}<\alpha<1$. Furthermore, we also show that the join $K_m\vee P_n$ is determined by its $A_{\alpha}$-spectrum for $\frac{1}{2}<\alpha<1$. In the end, we pose some related open problems for future study.


\end{abstract}

\section{Introduction}
Let $G=(V(G),E(G))$ be a graph with the vertex set $V(G)=\{v_1,v_2,\ldots,v_n\}$ and the edge set $E(G)=\{e_1,e_2,\ldots,e_m\}$. The \emph{adjacency matrix} of $G$, denoted by $A(G)=(a_{ij})_{n\times n}$, is an $n\times n$
symmetric matrix such that $a_{ij}=1$ if vertices $v_i$ and $v_j$
are adjacent and $0$ otherwise. Let $d_{i}=d_i(G)=d(v_i)=d_G(v_i)$ be the degree of
vertex $v_i$ in $G$. The \emph{degree matrix} of $G$, denoted by $D(G)$, is
the diagonal matrix with diagonal entries the vertex degrees of $G$. The \emph{Laplacian matrix} and the \emph{signless Laplacian matrix} of $G$ are defined as
$L(G)=D(G)-A(G)$ and $Q(G)=D(G)+A(G)$, respectively.

Nikiforov \cite{VN1} proposed to study the following matrix:
$$
A_{\alpha}(G)=\alpha D(G)+(1-\alpha)A(G),
$$
where $\alpha\in [0,1]$ is a real number. Note that $A_0(G)=A(G)$ and $2A_{1/2}(G)=Q(G)$. So, it was  claimed in \cite{VN1, VN2} that the matrices $A_{\alpha}(G)$ can underpin a unified theory of $A(G)$ and $Q(G)$. Up until now, a few properties on $A_{\alpha}(G)$ have been investigated, including bounds on the $k$-th largest  (especially, the largest, the second largest and the smallest)  eigenvalue of $A_{\alpha}(G)$ \cite{Lin, VN1, VN2, VN3}, the positive semidefiniteness of $A_{\alpha}(G)$ \cite{VN1, VN3}, etc. For more properties on $A_{\alpha}(G)$, we refer the reader to \cite{VN1}.

Let $M$ be an $n \times n$ real matrix. Denote by
$$
P_M(x)=\det(xI_n-M),
$$
the \emph{characteristic polynomial} of $M$, where $I_n$ is the identity matrix of size $n$. Denote the eigenvalues of $M$ by $\lambda_1(M)\geq\lambda_2(M)\geq\cdots\geq\lambda_n(M)$.
The collection of eigenvalues of $M$ together with multiplicities are called the \emph{spectrum} of $M$, denoted by $Spec(M)$. If $M=A_{\alpha}(G)$ (respectively, $A(G)$, $L(G)$, or $Q(G)$), then we simply write the \emph{spectrum} of $A_{\alpha}(G)$  (respectively, $A(G)$, $L(G)$, or $Q(G)$)  as \emph{$A_{\alpha}$-spectrum  } (respectively, \emph{$A$-spectrum,} \emph{$L$-spectrum}, or \emph{$Q$-spectrum}). Two graphs are said to be \emph{$M$-cospectral} if they have the same $M$-spectrum (equivalently, the same $M$-characteristic polynomial). A graph is called an \emph{$M$-DS graph} if it is \emph{determined by its $M$-spectrum}, meaning that there exists no other graph that is non-isomorphic to it but $M$-cospectral with it.

Characterizing which graphs are determined by their spectra is a classical but difficult problem in spectral graph theory which was raised by G\"{u}nthard and Primas \cite{GP} in 1956 with motivations from chemistry. Up until now, although many graphs have been proved to be DS graphs (see \cite{DE,DamH09}), the problem of determining DS  graphs is still far from being completely solved. In \cite[Concluding remarks]{DE}, van Dam and Haemers proposed to solve the following problem, where $J$ denotes the matrix with all entries equal to one:

\begin{problem}\label{prob:1}
Which linear combination of $D(G)$, $A(G)$, and $J$ gives the most DS graphs?
\end{problem}

From \cite[Table 1]{DE}, van Dam and Haemers claimed that the signless Laplacian matrix $Q(G)=D(G)+A(G)$ would be a good candidate. Since then, a lot of researchers tried to confirm this claim (see \cite{DamH09, DK, LiuLW11, LiuSD14, Liu, LiuWZY11, OG, WangHBM10, ZhangLZY09, ZhouB12} for example). When it comes to $A_{\alpha}(G)=\alpha D(G)+(1-\alpha)A(G)$, by enumerating the $A_{\alpha}$-characteristic polynomials for all graphs on at most 10 vertices (see \cite[Table 1]{LiuLiu17}), it seems that $A_{\alpha}$-spectra (especially, $\alpha>\frac{1}{2}$) are much more efficient than $Q$-spectra when we use them to distinguish graphs. In this paper, we show some graphs are determined by their $A_{\alpha}$-spectra, no mater which are or are not determined by $A$-, $L$- or $Q$-spectra. This in some sense supports the claim that $A_{\alpha}$-spectra are much more efficient than $Q$-spectra when we use them to distinguish graphs.

The rest of the paper is organized as follows. In Section \ref{sec:2}, we give some $A_{\alpha}$-DS graphs with $\alpha\in [0,1]$. In Section \ref{sec:3}, we prove that if $G$ is a regular graph, then $G$ is determined by its $A_{\alpha}$-spectrum if and only if the join $G\vee K_m$ is determined by its $A_{\alpha}$-spectrum for $\frac{1}{2}<\alpha<1$.  In particular, we prove that the join of a path and a complete graph is determined by their $A_{\alpha}$-spectra with $\alpha\in (1/2,1)$. In Section 4, we give some related open problems.

\section{Graphs determined by their $A_{\alpha}$-spectra with $\alpha\in [0,1]$}\label{sec:2}

For any graph, we can get a lot of information about its structure by its $A_{\alpha}$-spectrum. The following result presents some basic properties if two graphs have the same $A_{\alpha}$-spectra.
\begin{theorem}\label{th11}
Let $G$ and $G'$ be two graphs. If $Spec(A_\alpha(G))=Spec(A_\alpha(G'))$ with $\alpha\in [0,1]$, then we have the following statements:
\begin{itemize}
  \item[\rm (I)] $|V(G)|=|V(G')|$;
  \item[\rm (II)] $|E(G)|=|E(G')|$;
  \item[\rm (III)] If $G$ is $r$-regular, then $G'$ is $r$-regular;
\end{itemize}
Suppose that $d_1\geq d_2 \geq \cdots  \geq d_n$ and $d'_1\geq d'_2 \geq \cdots  \geq d'_n$ are the degree sequences of $G$ and $G'$, respectively. If $Spec(A_\alpha(G))=Spec(A_\alpha(G'))$ with $\alpha\in (0,1]$, then
\begin{itemize}
  \item[\rm (IV)] $\sum_{1\leq i<j\leq n}d_id_j=\sum_{1\leq i<j\leq n}d'_id'_j$;
  \item[\rm (V)] $\sum_{1\leq i\leq n}d^2_i=\sum_{1\leq i\leq n}{d'}_{i}^2$.
\end{itemize}
\end{theorem}

\begin{proof}
The statement \rm{(I)} is trivial, (II) and (V) follow from \cite[Propositions 34 and 35]{VN1}. Suppose that the characteristic polynomials of $G$ and $G'$ are
$$P_{A_\alpha (G)}(\lambda)=|\lambda I-A_\alpha (G)|=\lambda^n+ a_1\lambda^{n-1}+a_2\lambda^{n-2}+\cdots +a_{n-1}\lambda+a_n$$
and
$$P_{A_\alpha (G')}(\lambda)=|\lambda' I-A_\alpha (G')|=\lambda^n+ a'_1\lambda^{n-1}+a'_2\lambda^{n-2}+ \cdots +a'_{n-1}\lambda+a'_n .$$
Since $a_{2}=a'_{2}$, we have
\[
\sum_{1\leq i<j\leq n}\left| \begin{array}{cc}
  \alpha d_i & (1-\alpha)a_{i,j} \\
  (1-\alpha)a_{j,i}  & \alpha d_j
\end{array}\right|=\sum_{1\leq i<j\leq n}\left| \begin{array}{cc}
  \alpha d_i & (1-\alpha)a_{i,j} \\
  (1-\alpha)a_{j,i}  & \alpha d_j
\end{array}\right|,
\]
that is,
$$\alpha^2\sum_{1\leq i<j\leq n}d_id_j-(1-\alpha)^2\sum_{1\leq i<j\leq n}a^2_{ij}=\alpha^2\sum_{1\leq i<j\leq n}d'_id'_j-(1-\alpha)^2\sum_{1\leq i<j\leq n}a'^2_{ij}.$$
Note that $$\sum_{1\leq i<j\leq n}a^2_{ij}= |E(G)|=|E(G')|=\sum_{1\leq i<j\leq n}a'^2_{ij}.$$ Then \textsc{(IV)} holds.

Note that the average row sum of $A_{\alpha}(G)$ is $\frac{2|E(G)|}{n}$. If $G$ is $r$-regular, then $\lambda_1(A_{\alpha}(G))=r=\frac{2|E(G)|}{n}$, hence $\lambda_1(A_{\alpha}(G'))=r=\frac{2|E(G')|}{n}$. This implies that the largest eigenvalue of $A_{\alpha}(G')$ is equal to its average row sum. Thus, $G'$ is also $r$-regular, leading to (III).\hspace*{\fill}$\Box$
\end{proof}

Let $G$ and $H$ be two disjoint graphs. Denote by $G\cup H$ the \emph{disjoint union} of $G$ and $H$. Especially, $mG$ means the disjoint union of $m$ copies of $G$. The \emph{complement} of a graph $G$, denoted $\overline{G}$, is the graph with the same vertex set as $G$ such that two vertices are adjacent in $\overline{G}$ if and only if they are not adjacent in $G$. Using Theorem \ref{th11}, we deduce the following results.

\begin{theorem}
The following graphs are determined by their $A_\alpha$-spectrum:
\begin{itemize}
  \item[\rm (a)] the complete graph $K_n$;
  \item[\rm (b)] the star $K_{1,n-1}$ for $0<\alpha\leq 1$;
  \item[\rm (c)] the path $P_n$ for $0\leq\alpha<1$;
  \item[\rm (d)] the disjoint union of cycles $\bigcup^s_{i=1}C_{n_i}$ for $0\leq\alpha<1$;
  \item[\rm (e)] the complement of the disjoint union of cycles $\overline{\bigcup^s_{i=1}C_{n_i}}$ for $0\leq\alpha<1$;
  \item[\rm (f)] $kK_2\bigcup (n-2k)K_1$ where $1\leq k\leq \lfloor\frac{n}{2}\rfloor$ and $0\leq\alpha\leq1$;
  \item[\rm (g)] $\overline{kK_2\bigcup (n-2k)K_1}$ where $1\leq k\leq \lfloor\frac{n}{2}\rfloor$ and $0\leq\alpha\leq1$.
\end{itemize}
\end{theorem}

\begin{proof}
(a) According to Theorem \ref{th11} \textsc{(III)}, it follows that the complete graph is determined by its $A_\alpha$-spectrum.

(b) Let $G$ be $A_\alpha$-cospectral with $K_{1,n-1}$ for $0<\alpha\leq 1$. If $\alpha=1$, then the degree sequence of $G$ is $(n-1,1,\ldots,1)$. Thus $G\cong K_{1,n-1}$. If $0<\alpha<1$, then 0 is not the eigenvalue of $K_{1,n-1}$ and $\lambda_{2}(A_{\alpha}(K_{1,n-1}))=\alpha$ by \cite[Proposition 38]{VN1}. Thus $G$ have no isolated vertex. If  $G$ is not connected, then there exists at least two connected components, denoted  by  $U$ and $W$, then we have $\lambda_1(A_{\alpha}(U))\geq 1$  and $\lambda_1(A_{\alpha}(W))\geq 1$. It follows that $$\lambda_2(A_{\alpha}(G))\geq 1 > \alpha=\lambda_{2}(A_{\alpha}(K_{1,n-1})),$$ a contradiction. Hence, $G$ is connected. Then $G$ is a tree due to Theorem \ref{th11} \textsc{(II)}. Moreover, $K_{1,n-1}$ is the unique tree with maximal $A_{\alpha}$-spectral radius among all trees by \cite[Theorem 2]{VN2}. Therefore, $G\cong K_{1,n-1}$.

(c) Let $G$ be $A_\alpha$-cospectral with $P_{n}$. Note that $P_{n}$ is the unique graph with minimal $A_{\alpha}$-spectral radius among all connected graphs by \cite[Theorem 3]{VN2}. Thus we have $G\cong P_{n}$ if $G$ is connected. If $G$ is not connected, then by Theorem \ref{th11} \textsc{(II)}, there exists at least one component $U$ of $G$ containing cycles. This implies that $\frac{2|E(U)|}{|V(U)|}\geq 2$. So,
$$\lambda_{1}(A_{\alpha}(G))\geq \lambda_1(A_{\alpha}(U))\geq\frac{2|E(U)|}{|V(U)|}\geq 2,$$
it contradicts to the fact that $\lambda_{1}(A_{\alpha}(G))=\lambda_{1}(A_{\alpha}(P_{n}))<2$. Thus, $P_n$ is determined by its  $A_\alpha$-spectrum.

(d \& e) Suppose that $G$ is $r$-regular. Then $\lambda_1(A_{\alpha}(G))=r$ and $\lambda_i(A_{\alpha}(G))=\alpha r+(1-\alpha)\lambda_i(A).$
Thus, an $r$-regular graph is determined by its $A$-spectra, it is also determined by its $A_{\alpha}$-spectra. Note that the union of the cycles and the complement of the union of the cycles are both determined by its $A$-spectra. Then (d) and (e) hold.

(f) Let $G$ be $A_\alpha$-cospectral with $kK_2\bigcup (n-2k)K_1$. Let $U$ be a component of $G$ which is not an isolate vertex. We claim that $U\cong K_{2}$. If not,  we have $$\lambda_1(A_{\alpha}(G))>1=\lambda_1(A_{\alpha}(kK_2\bigcup (n-2k)K_1)).$$ Hence $G\cong sK_2\bigcup t K_1$. It is easy to see that $s=k$ and $t=n-2k$. Thus, $kK_2\bigcup (n-2k)K_1$ is determined by its $A_\alpha$-spectrum.

(g) Let $G$ be $A_\alpha$-cospectral with $\overline{kK_2\bigcup (n-2k)K_1}$. By Theorem \ref{th11} \textsc{(II)} we deduce that $G$ has at least $n-2k$ vertices with degree $n-1$. Let $n-1= \cdots= n-1\geq d_{1}\geq d_{2}\geq\cdots\geq d_{2k}$ be the degree sequence of $G$. By Theorem \ref{th11} \textsc{(II)} and \textsc{(IV)}, we have
$$(n-2k)(n-1)+\sum_{i=1}^{2k}d_{i}=(n-2k)(n-1)+2k(n-2)$$
and
$$(n-2k)(n-1)^{2}+\sum_{i=1}^{2k}d_{i}^{2}=(n-2k)(n-1)^{2}+2k(n-2)^{2}.$$
Then $\sum_{i=1}^{2k}d_{i}=2k(n-2)$ and $\sum_{i=1}^{2k}d_{i}^{2}=2k(n-2)^{2}$. According to Cauchy-Schwarz  inequality, it follows that
$$\sum_{i=1}^{2k}d_{i}^{2}\geq \frac{(\sum_{i=1}^{2k}d_{i})^{2}}{2k}=\frac{(2k(n-2))^{2}}{2k}=2k(n-2)^{2}.$$
Hence $d_{1}=\cdots=d_{2k}=n-2$. Thus $G\cong kK_2\bigcup (n-2k)K_1$, since $G$ has $n-2k$ vertices with degree $n-1$ and $2k$ vertices with degree $n-2$. This completes the proof.\hspace*{\fill}$\Box$
\end{proof}

\begin{remark}
{\em
The star $K_{1,n-1}$ is not determined by its $A_\alpha$-spectrum if $\alpha=0$. For example, both the star $K_{1,4}$ and $C_{4}\cup K_{1}$ have the same $A_0$-spectrum: $2, 0, 0, 0, -2.$
}
\end{remark}

In \cite{DM}, Doob and Haemers showed that the complement of the path is determined by its $A$-spectra. In the following, we prove that the complement of the path is also determined by its $A_{\alpha}$ spectrum when $0<\alpha<1$. Before proceeding, we need the following two lemmas.

\begin{lemma}\label{le3.1}
Let $G=\overline{C_{n_{1}}\cup C_{n_{2}}\cup H}$ and $G'=\overline{C_{n_{1}+n_{2}}\cup H}$ where $H$ is a graph of order $n-n_{1}-n_{2}$. Then we have $\lambda_{1}(A_{\alpha}(G))=\lambda_{1}(A_{\alpha}(G'))$.
\end{lemma}
\begin{proof}
We first show that $\lambda_{1}(A_{\alpha}(G'))\geq\lambda_{1}(A_{\alpha}(G))$. Let $X$ be the Perron vector of $A_{\alpha}(G)$. Suppose that $C_{n_{1}}=u_{1}u_{2}\ldots u_{n_{1}}u_{1}$ and $C_{n_{2}}=v_{1}v_{2}\ldots v_{n_{2}}v_{1}$. Clearly, $C_{n_{1}}\cup C_{n_{2}}-v_{1}v_{2}-u_{1}u_{2}+u_{1}v_{1}+u_{2}v_{2}$ is a cycle of order $n_{1}+n_{2}$. It is easy to see that $x_{u_1}=x_{u_{2}}$ and $x_{v_{1}}=x_{v_{2}}$. Hence
\begin{eqnarray*}
\lambda_{1}(A_{\alpha}(G'))-\lambda_{1}(A_{\alpha}(G))&\geq& 2(1-\alpha)(x_{v_{1}}x_{v_2}+x_{u_{1}}x_{u_{2}}-x_{u_{1}}x_{v_{1}}-x_{u_{2}}x_{v_{2}})\\
&=&2(1-\alpha)(x_{u_{1}}-x_{v_{1}})^2\\
&\geq&0.
\end{eqnarray*}
Then $\lambda_{1}(A_{\alpha}(G'))\geq\lambda_{1}(A_{\alpha}(G))$.

On the other hand, we assume that $X$ is the Perron vector of $A_{\alpha}(G')$ and let $C_{n_{1}+n_{2}}=v_{1}v_{2}\ldots v_{n_{1}+n_{2}}v_{1}$. Note that $$C_{n_{1}+n_{2}}+v_{1}v_{n_{1}}+v_{n_{1}+1}v_{n_{1}+n_{2}}-v_{1}v_{n_{1}+n_{2}}-v_{n_{1}}v_{n_{1}+1}\cong C_{n_{1}}\cup C_{n_{2}}.$$
Moreover, since $x_{v_{i}}=x_{v_{j}}$ for any $1\leq i,j\leq n_{1}+n_{2}$, we have
\begin{eqnarray*}
\lambda_{1}(A_{\alpha}(G))-\lambda_{1}(A_{\alpha}(G'))&\geq& 2(1-\alpha)(x_{v_{1}}x_{v_{n_{1}+n_{2}}}+x_{v_{n_{1}}}x_{v_{n_{1}+1}}-x_{v_{1}}x_{v_{n_{1}}}-x_{v_{n_{1}+1}}x_{v_{n_{1}+n_{2}}})\\
&=&0.
\end{eqnarray*} It implies that $\lambda_{1}(A_{\alpha}(G))\geq\lambda_{1}(A_{\alpha}(G')).$
This completes the proof.\hspace*{\fill}$\Box$
\end{proof}

\begin{lemma}\label{le3.2}
Let $G=\overline{P_{k}\cup C_{n-k}}$ with $2\leq k\leq n-3$. Then
\begin{itemize}
  \item[\rm (a)] $\lambda_{1}(A_{\alpha}(\overline{P_{n}}))>\lambda_{1}(A_{\alpha}(G))$, if $k$ is even;
  \item[\rm (b)] $\lambda_{1}(A_{\alpha}(\overline{P_{n}}))<\lambda_{1}(A_{\alpha}(G))$, if $k$ is odd.
\end{itemize}
\end{lemma}

\begin{proof}
(a) If $k$ is even, then we take $k=2s$. Assume that $P_{k}=v_{1}v_{2}\ldots v_{2s}$ and $C_{n-k}=u_{1}\ldots u_{n-k}u_{1}$. Let $X$ be the Perron vector of $A_{\alpha}(G)$. Note that
$$P_{k}\cup C_{n-k}+v_{s}u_{1}+v_{s+1}u_{2}-v_{s}v_{s+1}-u_{1}u_{2}\cong P_{n}.$$
By symmetry, we have $x_{v_{s}}=x_{v_{s+1}}$ and $x_{u_{1}}=x_{u_{2}}$, then
\begin{eqnarray*}
\lambda_{1}(A_{\alpha}(\overline{P_{n}}))-\lambda_{1}(A_{\alpha}(G))&\geq& 2(1-\alpha)(x_{v_{s}}x_{v_{s+1}}+x_{u_{1}}x_{u_{2}}-x_{u_{1}}x_{v_{s}}-x_{u_{2}}x_{v_{s+1}})\\
&=&2(1-\alpha)(x_{v_{s}}^2+x_{u_{1}}^2-2x_{u_{1}}x_{v_{s}})\\
&=&2(1-\alpha)(x_{v_{s}}-x_{u_{1}})^2\\
&\geq&0.
\end{eqnarray*}
Furthermore, if $x_{v_{s}}=x_{u_{1}}$, then it is easy to see that $x_{v_{s-1}}=x_{u_{1}}$ by eigenequations. By this way, we deduce that $x_{v_{1}}=\cdots=x_{v_{s}}=x_{u_{1}}$, a contradiction. This implies that $\lambda_{1}(A_{\alpha}(\overline{P_{n}}))>\lambda_{1}(A_{\alpha}(G))$.

(b) If $k$ is odd, then we take $k=2s+1$. Let $P_{n}=v_{1}v_{2}\ldots v_{n}$. Let $X$ be the Perron vector of $A_{\alpha}(\overline{P_{n}})$. Clearly, $x_{v_{i}}=x_{v_{n+1-i}}$ for all $1\leq i\leq n$. Since
$$P_{n}-v_{s}v_{s+1}-v_{n-s-1}v_{n-s}+v_{s}v_{n-s}+v_{s+1}v_{n-s-1}\cong P_{k}\cup C_{n-k},$$
we have
\begin{eqnarray*}
\lambda_{1}(A_{\alpha}(G))-\lambda_{1}(A_{\alpha}(\overline{P_{n}}))&\geq& 2(1-\alpha)(x_{v_{s}}x_{v_{s+1}}+x_{v_{n-s}}x_{n-s-1}-x_{v_{s+1}}x_{v_{n-s-1}}-x_{v_{s}}x_{v_{n-s}})\\
&=&2(1-\alpha)(x_{v_{s+1}}-x_{v_{n-s}})(x_{v_{s}}-x_{v_{n-s-1}})\\
&=&0~~(\text{since} ~~x_{v_{s+1}}=x_{v_{n-s}}).
\end{eqnarray*}
We can also confirm that $X$ is not an eigenvector of $\lambda_{1}(A_{\alpha}(G))$ by using the similar method to the case when $k$ is even. Thus, $\lambda_{1}(A_{\alpha}(G))>\lambda_{1}(A_{\alpha}(\overline{P_{n}}))$.\hspace*{\fill}$\Box$
\end{proof}

\begin{theorem}
The complement of the path ${P_{n}}$ is determined by its $A_{\alpha}$-spectrum for $0\leq \alpha<1$.
\end{theorem}
\begin{proof}
Recall \cite[Theorem 1]{DM} that the complement of the path ${P_{n}}$ is determined by its $A$-spectrum. So, it is sufficient to show that the theorem holds for $0<\alpha<1$. Consider an integer programming:
\[
\left\{
\begin{array}{l}
\min \sum_{i=1}^{n}a_{i}^{2}\\[0.1cm]
a_{i}\leq n-1\\[0.1cm]
\sum_{i=1}^{n}a_{i}=(n-2)(n-1).
\end{array}
\right.
\]
We claim that $(n-2,n-2,n-3,\ldots,n-3)$ is the unique optimal solution. If $a_{i}=n-1$ for some $i\in [n]$, then there exists $a_{j}\leq n-3$ for some $j\in [n]$. However, $(a_{i}-1)^2+(a_{j}+1)^2<a_{i}^2+a_{j}^2$, a contradiction. If $a_{i}\leq n-4$ for some $i\in [n]$, then there exists $a_{j}\geq n-2$ for some $j\in [n]$. Similarly, $(a_{i}-1)^2+(a_{j}+1)^2<a_{i}^2+a_{j}^2$, a contradiction. This implies that $a_{i}$ is equal to either $n-2$ or $n-3$. Thus, the claim holds.

Let $G$ be $A_\alpha$-cospectral with $\overline{P_{n}}$. According to Theorem \ref{th11} \textsc{(II)}, \textsc{(III)} and the integer programming, it follows that the degree sequence of $G$ is $(n-2,n-2,n-3,\ldots,n-3)$. Hence $G\cong \overline{P_{k}\cup C_{n-k}}$. If $k\neq n$, then by Lemmas \ref{le3.1} and \ref{le3.2}, we have either $\lambda_{1}(A_{\alpha}(G))>\lambda_{1}(A_{\alpha}(\overline{P_{n}}))$ or $\lambda_{1}(A_{\alpha}(G))<\lambda_{1}(A_{\alpha}(\overline{P_{n}}))$, a contradiction. Therefore, $G\cong \overline{P_{n}}$. This completes the proof.\hspace*{\fill}$\Box$
\end{proof}

\section{Graphs determined by their $A_{\alpha}$-spectra with $\alpha\in \left(\frac{1}{2},1\right)$}\label{sec:3}

\subsection{The $A_{\alpha}$-characteristic polynomial of a join}
The \emph{join} of two disjoint graphs $G$ and $H$, denoted by $G\vee H$, is the graph obtained by joining each vertex of $G$ to each vertex of $H$. In this subsection, we give the $A_{\alpha}$-characteristic polynomial of a join. Before proceeding, we give the definition of coronal of a matrix.

The \emph{$M$-coronal} of an $n\times n$ square matrix $M$, denoted by $\Gamma_M(x)$, is defined \cite{CuiT12, McLemanM11} to be the sum of the entries of the matrix $(xI_n-M)^{-1}$, that is,
$$\Gamma_M(x) = \mathbf{1}^T_n (xI_n-M)^{-1}\mathbf{1}_n,$$
where $\mathbf{1}_n$ denotes the column vector of size $n$ with all the entries equal to one, and $\mathbf{1}^T_n$ means the transpose of $\mathbf{1}_n$.

The following result is obtained by modifying \cite[Theorem 2.1]{Liu}.

\begin{lemma}\label{lem2.0}
Let $G_i$ be an arbitrary graph on $n_i$ vertices for $i = 1, 2.$ Then
$$P_{A_\alpha(G_1\vee G_2)}(x)=P_{A_{\alpha}(G_1)}(x-\alpha n_2)P_{A_{\alpha}(G_2)}(x-\alpha n_1)\left(1-(1-\alpha)^2\Gamma_{A_{\alpha}(G_1)}(x-\alpha n_2)\Gamma_{A_{\alpha}(G_2)}(x-\alpha n_1)\right).$$
\end{lemma}

\begin{proof}
Note that $A_{\alpha}(G_1\vee G_2)$ can be written as
\[A_{\alpha}(G_1\vee G_2)=\pmat{
                          \alpha n_2I_{n_1}+A_{\alpha}(G_1)  & (1-\alpha)J_{n_1\times n_2} \\[0.2cm]
                          (1-\alpha)J_{n_2\times n_1}  & \alpha n_1I_{n_2}+A_{\alpha}(G_2)
                       },\]
where $J_{s\times t}$ denotes the $s\times t$ matrix with all entries equal to one. Then
\begin{eqnarray*}
P_{A_{\alpha}(G_1\vee G_2)}(x)
&=& \det\pmat{
                          (x-\alpha n_2)I_{n_1}-A_{\alpha}(G_1)  & -(1-\alpha)J_{n_1\times n_2} \\[0.2cm]
                          -(1-\alpha)J_{n_2\times n_1}  & (x-\alpha n_1)I_{n_2}-A_{\alpha}(G_2)
                       }\\ [0.2cm]
&=& \det((x-\alpha n_2)I_{n_1}-A_{\alpha}(G_1))\cdot\det\left(S\right)\\ [0.2cm]
&=& P_{A_{\alpha}(G_1)}(x-\alpha n_2)\cdot\det(S),
\end{eqnarray*}
where
\begin{eqnarray*}
S&=&(x-\alpha n_1)I_{n_2}-A_{\alpha}(G_2)-(1-\alpha)^2J_{n_2\times n_1}((x-\alpha n_2)I_{n_1}-A_{\alpha}(G_1))^{-1}J_{n_1\times n_2}\\
&=&(x-\alpha n_1)I_{n_2}-A_{\alpha}(G_2)-(1-\alpha)^2\Gamma_{A_{\alpha}(G_1)}(x-\alpha n_2)J_{n_2\times n_2}
\end{eqnarray*}
is the Schur complement \cite{kn:Schur} of $(x-\alpha n_2)I_{n_1}-A_{\alpha}(G_1)$. Note that $J_{n_2\times n_2}=\mathbf{1}_{n_2}\mathbf{1}_{n_2}^T$, whose rank is equal to $1$. Then the result follows from
\begin{align*}
 \det (S)
&= P_{A_{\alpha}(G_2)}(x-\alpha n_1)\cdot \left(1-(1-\alpha)^2\Gamma_{A_{\alpha}(G_1)}(x-\alpha n_2)\cdot\mathbf{1}_{n_2}^T\left((x-\alpha n_1)I_{n_2}-A_{\alpha}(G_2)\right)^{-1}\mathbf{1}_{n_2}\right)\\
&= P_{A_{\alpha}(G_2)}(x-\alpha n_1)\cdot \left(1-(1-\alpha)^2\Gamma_{A_{\alpha}(G_1)}(x-\alpha n_2) \Gamma_{A_{\alpha}(G_2)}(x-\alpha n_1)\right).
\end{align*}
Here in the penultimate step we used the Sherman-Morrison formula for the determinant.\hspace*{\fill}$\Box$
\end{proof}

Up until now, a lot of $A$-, $L$- or $Q$-cospectal graphs are constructed by graphs operations (see \cite{CuiT12, CvetkovicDS95, Liu, LiuW12, LHH, McLemanM11, ZhangLY09, ZhouB12} for example). Here, by Theorem \ref{lem2.0}, we construct infinitely many pairs of $A_{\alpha}$-cospectral graphs, as stated in the following corollary (Note that, by setting $\alpha=0$ or $1/2$, we also obtain infinitely many pairs of $A$- or $Q$-cospectral graphs).

\begin{corollary}\label{Aacosp11}
\begin{itemize}
\item[\rm (a)] If $G$ is an arbitrary graph, and $H_1$ and $H_2$ are $A_{\alpha}$-cospectral graphs with $\Gamma_{A_{\alpha}(H_1)}(x)=\Gamma_{A_{\alpha}(H_2)}(x)$, then $G\vee H_1$ and $G\vee H_2$ are $A_{\alpha}$-cospectral.
\item[\rm (b)]   If $G_1$ and $G_2$ are $A_{\alpha}$-cospectral graphs with $\Gamma_{A_{\alpha}(G_1)}(x)=\Gamma_{A_{\alpha}(G_2)}(x)$, and $H_1$ and $H_2$ are $A_{\alpha}$-cospectral graphs with $\Gamma_{A_{\alpha}(H_1)}(x)=\Gamma_{A_{\alpha}(H_2)}(x)$, then $G_1\vee H_1$ and $G_2\vee H_2$ are $A_{\alpha}$-cospectral.
\end{itemize}
\end{corollary}

\begin{example}
{\em By Theorem \ref{th11} (I) and (III), we know that $A_{\alpha}$-cospectral regular graphs $H_1$ and $H_2$ must satisfy $\Gamma_{A_{\alpha}(H_1)}(x)=\Gamma_{A_{\alpha}(H_1)}(x)$ (see \cite[Fig. 2]{DE} for a pair of $A_{\alpha}$-cospectral regular graphs). Then, by Corollary \ref{Aacosp11} (a), we obtain that $G\vee H_{1}$ and $G\vee H_2$ are $A_{\alpha}$-cospectral, where $G$ is an arbitrary graph. By Corollary \ref{Aacosp11} (b), $\underbrace{H_1\vee H_1\vee \cdots \vee H_{1}}_i$ and $\underbrace{H_2\vee H_2\vee \cdots \vee H_2}_i$ for $i=1,2,3,\ldots$ are also $A_{\alpha}$-cospectral.}
\end{example}

It is known \cite[Proposition 2]{CuiT12} that, if $M$ is an $n \times n$ matrix with each row sum equal to a constant $t$, then
\begin{equation} \label{eq:GammaT}
\Gamma_{M}(x) = \frac{n}{x-t}.
\end{equation}
Let $G_i$ be an $r_i$-regular graph on $n_i$ vertices for $i=1,2$. Note that each row sum of $A_{\alpha}(G_i)$ is equal to $r_i$. Then, by (\ref{eq:GammaT}), we have
\begin{align}
  \Gamma_{A_{\alpha}(G_1)}(x-\alpha n_2) &= \frac{n_1}{x-\alpha n_2-r_1}, \label{eq:GammaTi1}\\[0.3cm]
 \Gamma_{A_{\alpha}(G_2)}(x-\alpha n_1) &= \frac{n_2}{x-\alpha n_1-r_2}. \label{eq:GammaTi2}
\end{align}
Substituting (\ref{eq:GammaTi1}) and (\ref{eq:GammaTi2}) into Lemma \ref{lem2.0}, we have the following corollary immediately.

\begin{corollary}\label{cor:regu}
Let $G_i$ be an $r_i$-regular
graph on $n_i$ vertices for $i = 1, 2$. Then
$$P_{A_\alpha(G_1\vee G_2)}(x)=\frac{P_{A_{\alpha}(G_1)}(x-\alpha n_2)P_{A_{\alpha}(G_2)}(x-\alpha n_1)}{(x-\alpha n_2-r_1)(x-\alpha n_1-r_2)}f(x),$$
where $f(x)=(x-\alpha n_2-r_1)(x-\alpha n_1-r_2)-n_1n_2.$
\end{corollary}

\subsection{The join of a complete graph and a regular graph}

Up until now, a lot of joins have been proved to be $A$-, $L$-, or $Q$-DS graphs (see \cite{ DamH09, DK, LiuLW11, LiuSD14, Liu, LiuW12, ZhangLY09, ZhouB12} for example). In this subsection, we prove that any join  of a complete graph and a regular graph is determined by the $A_{\alpha}$-spectrum with $\frac{1}{2}<\alpha<1$.

In \cite{Lin}, Lin, Xue and Shu gave the following result.

\begin{lemma}[\cite{Lin}, Theorem 1.2]\label{lem2.1}
Let $G$ be a graph with $n$ vertices and $1/2<\alpha<1$. Then
$\lambda_k(A_{\alpha}(G))= \alpha n-1 \ \mbox{for $k\geq2$}$ if and only if $G$ has $k$ vertices of degree $n-1$.
\end{lemma}

\begin{lemma}\label{DegVee1}
Let $G$ be an $r$-regular graph on $n$ vertices and $G$ is determined by its $A_{\alpha}$-spectrum. Let $H$ be a graph $A_{\alpha}$-cospectral with $G\vee K_m$. If $d_1(H)=d_2(H)=\cdots=d_m(H)=n+m-1$, then $H\cong G\vee K_m$.
\end{lemma}
\begin{proof}
The following equations follow from the fact that $A_{\alpha}$-cospectral graphs have the same sum of vertex degrees and the same sum of square of vertex degrees (see Theorem \ref{th11} (II) and (V)):
\begin{eqnarray}
  \sum_{i=1}^{m+n}d_i(H)     &=& m(n+m-1)+n(r+m),\label{VeeEqua1}  \\
   \sum_{i=1}^{m+n}d_i(H)^2  &=&  m(n+m-1)^2+n(r+m)^2.\label{VeeEqua2}
\end{eqnarray}
Plugging $d_1(H)=d_2(H)=\cdots=d_m(H)=n+m-1$ into Equations (\ref{VeeEqua1}) and (\ref{VeeEqua2}), we have
\begin{eqnarray*}
  \sum_{i=m+1}^{m+n}d_i(H)&=&n(r+m),  \\
   \sum_{i=m+1}^{m+n}d_i(H)^2&=&n(r+m)^2.
\end{eqnarray*}
Then
\begin{equation*}
  \sum_{i=m+1}^{m+n}(d_i(H)-r-m)^2=\sum_{i=m+1}^{m+n}d_i(H)^2-2(r+m)\sum_{i=m+1}^{m+n}d_i(H)+\sum_{i=m+1}^{m+n}(r+m)^2=0.
\end{equation*}
This implies that $d_{m+1}(H)=d_{m+2}(H)=\cdots=d_{m+n}(H)=r+m$. Then $H\cong G_1\vee K_m$, where $G_1$ is an $r$-regular graph. Corollary \ref{cor:regu} implies that $G$ and $G_1$ are $A_{\alpha}$-cospectral graph. Thus, $H\cong G\vee K_m$ comes from the assumption that $G$ is determined by its $A_{\alpha}$-spectrum.\hspace*{\fill}$\Box$  \end{proof}

\begin{theorem}\label{VeeQthmKm}
Let $G$ be an $r$-regular graph on $n$ vertices for $0\leq r\leq n-1$. If $1/2<\alpha<1$, then $G$ is determined by its $A_{\alpha}$-spectrum if and only if $G\vee K_m$ is also determined by its $A_{\alpha}$-spectrum.
\end{theorem}

\begin{proof}
If $G$ is determined by its $A_{\alpha}$-spectrum, assume that $H$ be $A_{\alpha}$-cospectral with $G\vee K_m$. Then Lemma \ref{lem2.1} implies that $d_1(H)=d_2(H)=\cdots=d_m(H)=m+n-1$. By Lemmas \ref{DegVee1}, we have $H\cong G\vee K_m$.

Conversely, if $G\vee K_m$ is determined by its $A_{\alpha}$-spectrum, assume that $G_1$ is $A_{\alpha}$-cospectral with $G$. By Theorem \ref{th11} (III), $G_1$ is an $r$-regular graph. Then, by Corollary \ref{cor:regu}, we have $P_{K_m\vee G}(x)=P_{K_m\vee G_1}(x)$, and then $K_m\vee G\cong K_m\vee G_1$ by the assumption that  $G\vee K_m$ is determined by its $A_{\alpha}$-spectrum. Thus, it follows that $G\cong G_1$.

This completes the proof.\hspace*{\fill}$\Box$
\end{proof}

In Theorem \ref{VeeQthmKm}, we proved that if an $r$-regular $G$ is determined by its $A_{\alpha}$-spectrum for $1/2<\alpha<1$, then $G\vee K_m$ is also determined by its $A_{\alpha}$-spectrum for $1/2<\alpha<1$. This can help us to find more graphs determined by their $A_{\alpha}$-spectra provided that we have found enough many regular graphs determined by their $A_{\alpha}$-spectra.  Next, we give many regular graphs determined by their $A_{\alpha}$-spectra.




\begin{proposition}\label{reg:case}
Let $G$ be an $r$-regular graph determined by its $A$-spectrum (respectively, $L$-spectrum, $Q$-spectrum). Then $G$ is determined by its $A_{\alpha}$-spectrum.
\end{proposition}

\begin{proof}
Suppose that $H$ and $G$ are $A_{\alpha}$-cospectral graphs. Then, by Theorem \ref{th11} \rm{(III)}, $H$ is an $r$-regular graph.  Note that the $A_{\alpha} $-spectrum of $H$ is $\alpha r+(1-\alpha) \lambda_i(G)$, where $\lambda_i(G)$ are adjacency eigenvalues of $G$. Then $H$ and $G$ are $A$-cospectral. Therefore, $H$ and $G$ are isomorphic since $G$ is  determined by its $A$-spectrum.

Similarly, we can verify that the result is also valid if $G$ is an $r$-regular graph determined by its $L$-spectrum or $Q$-spectrum, respectively.\hspace*{\fill}$\Box$
\end{proof}

\begin{remark}
{\em By Proposition \ref{reg:case}, we can obtain many regular graphs determined by their $A_{\alpha}$-spectra, since a lot of regular graphs have been proven to be determined by their $A$-spectra (respectively, $L$-spectra, $Q$-spectra) (see \cite[Sections 5, 6 and 7]{DE}). Then, by Theorem \ref{VeeQthmKm}, we can construct lots of joins determined by their $A_{\alpha}$-spectra (see Corollaries \ref{cor12:4}--\ref{cor12:3} for example).}
\end{remark}

Note \cite[Section 6.2]{CvetkovicDS95} that any $r$-regular graph on $n$ vertices for $r=0,1,2,n-3,n-2, n-1$ is determined by its $A$-spectrum. Proposition \ref{reg:case} implies that such an $r$-regular graph is also determined by its $A_{\alpha}$-spectrum. Then, by Theorem \ref{VeeQthmKm}, we have the following result immediately.

\begin{corollary}\label{cor12:4}
Let $G$ be an $r$-regular graph on $n$ vertices for $r=0,1,2,n-3,n-2, n-1$. Then the join $G\vee K_m$ is determined by its $A_{\alpha}$-spectra for $\frac{1}{2}<\alpha<1$.
\end{corollary}

If $G$ is the disjoint union of $K_1$, then Corollary \ref{cor12:4} implies the following result immediately.

\begin{corollary}\label{cor12:1}
The complete split graph is determined by its $A_{\alpha}$-spectra for $\frac{1}{2}<\alpha<1$.
\end{corollary}

Similarly, by setting $G$ be the disjoint union of $K_2$ and $K_m=K_1$ in Corollary \ref{cor12:4}, we have the following result immediately.

\begin{corollary}\label{cor12:2}
The friendship graph is determined by its $A_{\alpha}$-spectra for $\frac{1}{2}<\alpha<1$.
\end{corollary}

At last, by setting $G\cong C_n$ and $K_m=K_1$,  Corollary \ref{cor12:4} leads to the following result.

\begin{corollary}\label{cor12:3}
The wheel graph is determined by its $A_{\alpha}$-spectra for $\frac{1}{2}<\alpha<1$.
\end{corollary}

\subsection{The join of a complete graph and an irregular graph}

In this section, we first show $K_m\vee P_n$ is determined by its $A_{\alpha}$-spectra for $\frac{1}{2}<\alpha<1.$

\begin{theorem}\label{thm3.1}
Let $m,n\geq 1$. Then $K_m\vee P_n$ is determined by its $A_{\alpha}$-spectra for $\frac{1}{2}<\alpha<1.$
\end{theorem}
\begin{proof}
Let $G$ be $A_{\alpha}$-cospectral with $K_m\vee P_n$. Then by Lemma \ref{lem2.1}, we know that $G$ has $m$ vertices with degree $m+n-1.$
So we assume that $G\cong K_m\vee H$ with $|V(H)|=n.$ Suppose that $d_1(G)=d_2(G)=\cdots=d_m(G)=n+m-1$.
The following equations follow from the fact that $A_{\alpha}$-cospectral graphs have the same sum of vertex degrees and the same sum of square of vertex degrees (see Theorem \ref{th11} (II) and (V)):
\begin{eqnarray}
  \sum_{i=1}^{m+n}d_i(G)     &=& m(n+m-1)+2(m+1)+(n-2)(m+2),\label{VeeEqua1}  \\
   \sum_{i=1}^{m+n}d_i(G)^2  &=&  m(n+m-1)^2+2(m+1)^2+(n-2)(m+2)^2.\label{VeeEqua2}
\end{eqnarray}
Note that $d_i(G)=d_i(H)+m$ for $i=m+1,\ldots,n+m$, and plugging $d_1(H)=d_2(H)=\cdots=d_m(H)=n+m-1$ into Equations (\ref{VeeEqua1}) and (\ref{VeeEqua2}), we have
\begin{eqnarray}
  \sum_{i=m+1}^{m+n}d_i(H)&=&2n-2,\label{VeeEqua3}  \\
   \sum_{i=m+1}^{m+n}d_i(H)^2&=&4n-6.\label{VeeEqua4}
\end{eqnarray}
Then $(\ref{VeeEqua4})-4\times(\ref{VeeEqua3})$, we have \begin{eqnarray}\sum_{i=m+1}^{m+n}(d_i(H)-2)^2=2.\label{VeeEqua5}\end{eqnarray}

Combining Equations (\ref{VeeEqua3}) and (\ref{VeeEqua5}), we have that $H$ is a path, or $H$ is the disjoint union of cycles and a path. Suppose that $H$ is the disjoint union of cycles and a path. In the following, we will prove that
\begin{equation}\label{eq:noeq212}
\lambda_{1}(A_{\alpha}(K_m\vee H))>\lambda_{1}(A_{\alpha}(K_m\vee P_n)),
\end{equation}
which contradicts our assumption that $K_m\vee H$ and $K_m\vee P_n$ are $A_{\alpha}$-cospectral. Thus, $H$ is just a path, which implies that $K_m\vee P_n$ is determined by its $A_{\alpha}$-spectra for $\frac{1}{2}<\alpha<1.$

Now, we prove (\ref{eq:noeq212}). First, consider the following claim.

\medskip

\noindent{\textbf{Claim 1.}} Let $H'=C_{n_{1}}\cup C_{n_{2}}\cup H_1$ and $H''=C_{n_{1}+n_{2}}\cup H_1$, where $H_1$ is a graph of order $n-n_{1}-n_{2}$. Then
$$\lambda_{1}(A_{\alpha}(K_m\vee H'))=\lambda_{1}(A_{\alpha}(K_m\vee H'')).$$

\noindent{\emph{Proof of Claim 1.}} The proof is similar to that of Lemma \ref{le3.1}, and hence we omit the detail.

Without loss of generality, by Claim 1, we assume that $H=C_k\cup P_{n-k}$ with $3\leq k\leq n-1$.

\medskip

\noindent{\textbf{Claim 2.}} Let $H=C_k\cup P_{n-k}$ with $3\leq k\leq n-1$. Then $$\lambda_{1}(A_{\alpha}(K_m\vee H))>\lambda_{1}(A_{\alpha}(K_m\vee P_n)).$$

\noindent{\emph{Proof of Claim 2.}} Let $X$ be the Perron vector of $K_m\vee P_n$ and $P_n=\{v_1,\ldots,v_n\}$. By symmetry, assume that $X=(x,\ldots,x,x_1,\ldots,x_n)^t$ with $x_i=x_{n-i+1}$.
Note that
\[\left\{\begin{array}{l}
\lambda x_1=\alpha (m+1)x_1+(1-\alpha)x_2+(1-\alpha)mx,\\[0.1cm]
\lambda x_2=\alpha (m+2)x_2+(1-\alpha)x_1+(1-\alpha)x_3+(1-\alpha)mx,\\[0.1cm]
\lambda x_3=\alpha (m+2)x_3+(1-\alpha)x_2+(1-\alpha)x_4+(1-\alpha)mx,\\[0.1cm]
\vdots\\[0.1cm]
\lambda x_{\lfloor\frac{n}{2}\rfloor-1}=\alpha (m+2)x_{\lfloor\frac{n}{2}\rfloor-1} +(1-\alpha)x_{\lfloor\frac{n}{2}\rfloor-2}+(1-\alpha)x_{\lfloor\frac{n}{2}\rfloor} +(1-\alpha)mx.\\
\end{array}
\right.
\]
Then
\begin{equation}\label{Maineq1}
\left\{
\begin{array}{l}
\lambda_1 (x_2-x_1)=\alpha (m+2)(x_2-x_1)+ x_1+(1-\alpha)(x_3-x_2),\\[0.1cm]
\lambda_1 (x_{i+1}-x_i)=\alpha (m+2)(x_{i+1}-x_i)\\[0.1cm]
\hspace{2.8cm}+(1-\alpha)(x_i-x_{i-1})+(1-\alpha)(x_{i+2}-x_{i+1}) ~~ \mbox{for $2\leq i\leq \lfloor\frac{n}{2}\rfloor-1$}.
\end{array}
\right.
\end{equation}
By \cite[Corollary 12]{VN1}, we have
\begin{eqnarray}
\lambda_1(A_{\alpha}(K_m\vee P_n))>\alpha \Delta=\alpha (m+n-1)>\alpha(m+2)+(1-\alpha)\frac{n}{2}.\label{VeeEqua6}
\end{eqnarray}
%
%
By (\ref{Maineq1}), we have
$$
\big(\lambda_1-\alpha (m+1)+(1-\alpha)\big)(x_2-x_1)=\alpha x_2+(1-\alpha)x_3>0.
$$
This implies that $x_2>x_1$.

Now we are ready to prove the {\bf{Claim 2}}. Consider the following two cases.

\medskip

\noindent{\textbf{Case 1.}} $k\leq n-2.$

\medskip

Choose two vertices $v_i$ and $v_{i+k}$ such that
$$
\left\{
\begin{array}{ll}
  |d(v_i,v_{\frac{n+1}{2}})-d(v_{i+k},v_{\frac{n+1}{2}})|\le 1,&  \text{if $n$ is odd}, \\[0.3cm]
  |d(v_i,v_{\frac{n}{2}})-d(v_{i+k},v_{\frac{n}{2}+1})|\le 1,&  \text{if $n$ is even},
\end{array}
\right.
$$
where $d(v_i,v_j)$ denotes the distance between $v_i$ and $v_j$. It is easy to see that
$$
K_m\vee H=K_m\vee P_n-v_{i-1}v_{i}-v_{i+k}v_{i+k+1}
+v_{i-1}v_{i+k+1}+v_{i}v_{i+k}.
$$
Then
\begin{align*}
\lambda_1(A_{\alpha}(K_m\vee H))&-\lambda_1(A_{\alpha}(K_m\vee P_n))\\
&\geq2(1-\alpha)( x_{i-1}x_{i+k+1}+x_{i}x_{i+k}-x_{i-1}x_i-x_{i+k}x_{i+k+1})\\
&=2(1-\alpha)(x_{i+k}-x_{i-1})(x_{i}-x_{i+k+1}).
\end{align*}

Suppose that $n$ is odd. Then $k=2(\frac{n+1}{2}-i)-1,2(\frac{n+1}{2}-i)$ or $2(\frac{n+1}{2}-i)+1$. If $k=2(\frac{n+1}{2}-i)$, then $x_i=x_{i+k}$ and $x_{i-1}=x_{i+k+1}$, thus
\begin{align*}
\lambda_1(A_{\alpha}(K_m\vee H))&-\lambda_1(A_{\alpha}(K_m\vee P_n))\\
&\geq2(1-\alpha)(x_{i+k}-x_{i-1})^2\\
&\geq0.
\end{align*}
Otherwise,
\begin{align*}
\lambda_1(A_{\alpha}(K_m\vee H))&-\lambda_1(A_{\alpha}(K_m\vee P_n))\\
&\geq2(1-\alpha)(x_{i+k}-x_{i-1})(x_{i}-x_{i+k+1})\\
&=0
\end{align*}
since $x_i=x_{i+k+1}$ when $k=2(\frac{n+1}{2}-i)-1$ and $x_{i-1}=x_{i+k}$ when $k=2(\frac{n+1}{2}-i)+1$.

Suppose that $n$ is even. The discussion is completely similar to the case when $n$ is odd.
Thus, we conclude that
$$
\lambda_1(A_{\alpha}(K_m\vee H))\geq\lambda_1(A_{\alpha}(K_m\vee P_n)).
$$

Suppose that $\lambda_1(A_{\alpha}(K_m\vee H))=\lambda_1(A_{\alpha}(K_m\vee P_n))$.
Then, $X$ is also the Perron vector of $A_{\alpha}(K_m\vee H)$
and $x_i=x_{i+1}=\cdots=x_{i+k}$ by symmetry (Recall that $H=C_k\cup P_{n-k}$ with $3\leq k\leq n-1$). Since $k\geq3$, by (\ref{Maineq1}), we obtain that
\begin{align*}
(1-\alpha)(x_i-x_{i-1})&=[\lambda_1- \alpha (m+2)](x_{i+1}-x_i)-(1-\alpha)(x_{i+2}-x_{i+1})=0,\\[0.1cm]
(1-\alpha)(x_{i-1}-x_{i-2})&=[\lambda_1- \alpha (m+2)](x_i-x_{i-1})-(1-\alpha)(x_{i+1}-x_i)=0,\\[0.1cm]
\vdots\\
(1-\alpha)(x_2-x_1)&=[\lambda_1- \alpha (m+2)](x_3-x_2)-(1-\alpha)(x_4-x_3)=0,
\end{align*}
i.e., $x_1=x_2=\cdots=x_{i-1}=x_i$, which  contradicts to $x_2>x_1$. Thus,
$$\lambda_1(A_{\alpha}(K_m\vee H))>\lambda_1(A_{\alpha}(K_m\vee P_n)).$$

\noindent{\textbf{Case 2.}} $k=n-1.$

\medskip

Note that $K_m\vee H=K_m\vee P_n-v_1v_2+v_2v_n.$ Then
\begin{align*}
\lambda_1(A_{\alpha}(K_m\vee H))&-\lambda_1(A_{\alpha}(K_m\vee P_n))\\
&\geq 2(1-\alpha)(x_2x_n-x_1x_2)+\alpha(x_n^2-x_1^2)\\
&=0.
\end{align*}
Suppose that $\lambda_1(A_{\alpha}(K_m\vee H))=\lambda_1(A_{\alpha}(K_m\vee P_n))$. Then, $X$ is also the Perron vector of $A_{\alpha}(K_m\vee H)$
and $x_2=x_3=\cdots=x_n$ by symmetry. By (3.12), we obtain that
$$(1-\alpha)(x_2-x_1)=[\lambda_1- \alpha (m+2)](x_3-x_2)-(1-\alpha)(x_4-x_3)=0,$$
which  contradicts to $x_2>x_1$. Thus
$$\lambda_1(A_{\alpha}(K_m\vee H))>\lambda_1(A_{\alpha}(K_m\vee P_n)).$$
This completes the proof of {\bf{Claim 2}}.

Therefore, by {\bf{Claims 1 and 2}}, we conclude that (\ref{eq:noeq212}) is valid. This leads to that $K_m\vee P_n$ is determined by its $A_{\alpha}$-spectra.\hspace*{\fill}$\Box$
\end{proof}

\textrm{\begin{remark}
Very recently, for sufficiently large $n$, Tait and Tobin \cite{TT} showed that $K_2\vee P_{n-2}$ and $K_1\vee P_{n-1}$ attain the maximal spectral radius among all planar graphs and outerplanar graphs, respectively. As a corollary of Theorem \ref{thm3.1}, these two graphs are determined by their $A_{\alpha}$-spectrum for $1/2<\alpha<1$.
\end{remark}}

\section{Concluding remarks}
In Section \ref{sec:3}, we prove that the join of some graphs $G$ and $K_m$ are determined by their $A_{\alpha}$-spectra for $1/2<\alpha<1$ provided that $G$ is determined by its $A_{\alpha}$-spectrum for $1/2<\alpha<1$. In particular, if $G$ is an $r$-regular graph, then $G$ is determined by its $A_{\alpha}$-spectrum if and only if $G\vee K_m$ is determined by its $A_{\alpha}$-spectrum for $1/2<\alpha<1$. Motivated by these results, we pose the following problems:
\begin{problem}\label{Qes:1}
Characterizing more graphs $G$ determined by their $A_{\alpha}$-spectra such that $G\vee K_{m}$ are also determined by their $A_{\alpha}$-spectra for $\alpha\in(1/2,1)$.
\end{problem}


\begin{problem}\label{Qes:2}
Let $G$ be an $r$-regular graph determined by its $A_{\alpha}$-spectrum. Finding more graphs $H$ such that $G\vee H$ are determined by its $A_{\alpha}$-spectrum for $\alpha\in(1/2,1)$.
\end{problem}

We would also like to propose the following problem which could be regarded as an impetus to sloving Problem \ref{Qes:2}. Note that Corollaries \ref{cor12:1}, \ref{cor12:2} and \ref{cor12:3} are also the motivation why we pose the following problem.

\begin{problem}\label{Qes:3}
Let $G$ be an $r$-regular graph determined by its $A_{\alpha}$-spectrum. Is $G\vee (nK_1)$ determined by its $A_{\alpha}$-spectrum for $\alpha\in(1/2,1)$?
\end{problem}

Problems \ref{Qes:1}, \ref{Qes:2} and \ref{Qes:3} are investigated under the condition $\alpha\in(1/2,1)$. So, it is natural to study them under the condition $\alpha\in(0,1/2)$.

\begin{problem}\label{Qes:4}
Investigating Problems \ref{Qes:1}, \ref{Qes:2} and \ref{Qes:3} under the condition $\alpha\in(0,1/2)$.
\end{problem}



\small {

}

\end{document}